\documentclass[paper=a4,english,fontsize=11pt,parskip=half,abstract=true]{scrartcl}
\usepackage{babel}
\usepackage[utf8]{inputenc}
\usepackage[T1]{fontenc}
\usepackage[left=20mm,right=20mm,top=30mm,bottom=30mm]{geometry}
\usepackage{amsmath}
\usepackage{amsthm}
\usepackage{amssymb}
\usepackage{color}
\usepackage{enumerate} 
\usepackage{thmtools}
\usepackage{mathtools}
\mathtoolsset{centercolon} 
\usepackage[bookmarks=true,
            pdftitle={On the converse of Gaschütz' complement theorem},
            pdfauthor={Benjamin Sambale},
            pdfkeywords={},
            pdfstartview={FitH}]{hyperref}

\newtheorem{Thm}{Theorem} 
\newtheorem{Lem}[Thm]{Lemma}
\newtheorem{Prop}[Thm]{Proposition}
\newtheorem{Cor}[Thm]{Corollary}

\newtheorem{Prob}[Thm]{Problem}
\theoremstyle{definition}

\numberwithin{equation}{section}

\allowdisplaybreaks[1]

\renewcommand{\phi}{\varphi}
\newcommand{\C}{\mathrm{C}}
\newcommand{\N}{\mathrm{N}}
\newcommand{\Z}{\mathrm{Z}}
\newcommand{\pcore}{\mathrm{O}}

\newcommand{\NN}{\mathbb{N}}

\newcommand{\Aut}{\mathrm{Aut}}
\newcommand{\Inn}{\mathrm{Inn}}
\newcommand{\Out}{\mathrm{Out}}
\newcommand{\GL}{\mathrm{GL}}
\newcommand{\SL}{\mathrm{SL}}

\newcommand{\Syl}{\operatorname{Syl}}

\title{On the converse of Gaschütz'\\ complement theorem}
\author{Benjamin Sambale\footnote{Institut für Algebra, Zahlentheorie und Diskrete Mathematik, Leibniz Universität Hannover, Welfengarten 1, 30167 Hannover, Germany,
\href{mailto:sambale@math.uni-hannover.de}{sambale@math.uni-hannover.de}}
}
\date{\today\footnote{This version differs significantly from the published version in J. Group Theory. The proofs of Šemetkov's theorem and Yonaha's theorem were greatly simplified.}}

\begin{document}
\frenchspacing
\maketitle
\begin{abstract}\noindent
Let $N$ be a normal subgroup of a finite group $G$. Let $N\le H\le G$ such that $N$ has a complement in $H$ and $(|N|,|G:H|)=1$. If $N$ is abelian, a theorem of Gaschütz asserts that $N$ has a complement in $G$ as well. Brandis has asked whether the commutativity of $N$ can be replaced by some weaker property.
We prove that $N$ has a complement in $G$ whenever all Sylow subgroups of $N$ are abelian. 
On the other hand, we construct counterexamples if $\Z(N)\cap N'\ne 1$. For metabelian groups $N$, the condition $\Z(N)\cap N'=1$ implies the existence of complements. 
Finally, if $N$ is perfect and centerless, then Gaschütz' theorem holds for $N$ if and only if $\Inn(N)$ has a complement in $\Aut(N)$. 
\end{abstract}


\renewcommand{\sectionautorefname}{Section}
\section{Introduction}

It is a difficult problem to classify all finite groups $G$ with a given normal subgroup $N$ and a given quotient $G/N$. The situation becomes much easier if $N$ has a \emph{complement} $H$ in $G$, i.\,e. $G=HN$ and $H\cap N=1$. Then $G$ is determined by the conjugation action $H\to\Aut(N)$ and $G\cong N\rtimes H$.
A well-known theorem by Schur asserts that $N$ always has a complement if $N$ is abelian and $\gcd(|N|,|G:N|)=1$. Zassenhaus~\cite[Theorem~IV.7.25]{ZassenhausBook} observed that the statement holds even without the commutativity of $N$ (now called the Schur--Zassenhaus theorem). 
Although we are only interested in the existence of complements, we mention that all complements in this situation are conjugate in $G$ by virtue of the Feit--Thompson theorem.
In 1952, Gaschütz~\cite[Satz~1 on p. 99]{Gaschutz} (see also \cite[Hauptsatz~III.17.4]{Huppert}) found a way to relax the coprime condition in Schur's theorem as follows.

\begin{Thm}[\textsc{Gaschütz}]\label{Ga}
Let $N$ be an abelian normal subgroup of a finite group $G$. Let $N\le H\le G$ such that $N$ has a complement in $H$ and $\gcd(|N|,|G:H|)=1$. Then $N$ has a complement in $G$.
\end{Thm}

Unlike Schur's theorem, \autoref{Ga} does not generalize to non-abelian groups $N$. The counterexample of smallest order is attributed to Baer, see \cite[p.~225]{ScottBook}. In modern notation it can be described as a central product $G=\SL(2,3)*C_4$ ($=\mathtt{SmallGroup}(48,33)$ in the small groups library~\cite{GAP48}) where $N=Q_8\unlhd G$ has a complement in a Sylow $2$-subgroup $H=Q_8*C_4\le G$, but not in $G$ (here, $Q_8$ denotes the quaternion group of order $8$ and $C_4$ is a cyclic group of order $4$). A similar counterexample, given by Hofmann~\cite[pp. 32--33]{Hofmann} and reproduced in Huppert's book~\cite[Beispiel~I.18.7]{Huppert}, has order $|G|=|G:H||H:N||N|=2\cdot 3^2\cdot 3^3$. Finally, a more complicated counterexample of order $2^73^2$ is outlined in Zassenhaus' book~\cite[Appendix~F, Exercise~5]{ZassenhausBook}). We produce more general families of counterexamples in \autoref{sec3}.

Although Gaschütz did obtain some non-abelian variations of his theorem, he confesses:
\begin{quote}
\textit{“Ihre Verallgemeinerung auf nichtabelsche Erweiterungen ist mir bisher nicht gelungen.”}\footnote{Translation: \emph{I did not yet succeed with their} [his theorems] \emph{generalization to non-abelian extensions.}}
\end{quote}

Brandis~\cite{Brandis} not only gave a very elementary proof of \autoref{Ga} (avoiding cohomology), but also replaced abelian groups by solvable groups under further technical assumptions.
However, he concludes like Gaschütz with:
\begin{quote}
\textit{“Insbesondere wäre es interessant zu wissen: gibt es eine größere Klasse $\mathcal{R}$ von Gruppen, als die Klasse der abelschen Gruppen, so daß falls $A\in\mathcal{R}$ folgt: der Satz von Gaschütz ist für $A$ anwendbar.”}\footnote{Translation: \emph{In particular, it would be of interest to know: is there a bigger class $\mathcal{R}$ of groups, than the class of all abelian groups, such that if $A\in\mathcal{R}$, then Gaschütz' theorem applies to $A$.}}
\end{quote}

Following his words we say that \emph{Gaschütz' theorem holds for} $N$ if for every embedding $N\le H\le G$ such that $N\unlhd G$, $\gcd(|N|,|G:H|)=1$ and $N$ has a complement in $H$, then $N$ has a complement in $G$.
By analogy to the notation of \emph{control of fusion/transfer} one could say that $H$ \emph{controls complements} of $N$ in $G$. 
Our main theorem generalizes Gaschütz' theorem as follows:

\begin{Thm}\label{sylow}
Suppose that all Sylow subgroups of $N$ are abelian. Then Gaschütz' theorem holds for $N$.
\end{Thm}

The proof of \autoref{sylow} follows easily from a forgotten theorem of Šemetkov\footnote{the transliteration Shemetkov is also used in the literature}, which is presented and proved in the next section. We further collect and prove numerous other theorems on the existence of complements in that section. In the last section we construct counterexamples to show the following among other theorems:

\begin{Thm}
If $N'\cap\Z(N)\ne 1$, then Gaschütz' theorem does not hold for $N$. In particular, it does not hold for nonabelian nilpotent groups.
\end{Thm}

\section{On the existence of complements}\label{sec2}

Our notation is standard apart from the commutator convention $[x,y]:=xyx^{-1}y^{-1}$ for elements $x,y$ of a group. The commutator subgroup, the center and the Frattini subgroup of $G$ are denoted by $G'=[G,G]$, $\Z(G)$ and $\Phi(G)$ respectively. For $H\le G$ and $x\in G$ we write $^xH=xHx^{-1}$.
We will often use the following elementary fact: If $K$ is a complement of $N$ in $G$ and $N\le H\le G$, then $H\cap K$ is a complement of $N$ in $H$. Indeed, by the Dedekind law we have $N(H\cap K)=NK\cap H=H$ and $(H\cap K)\cap N=1$. The same argument shows that $KM/M$ is a complement of $N/M$ in $G/M$ for every normal subgroup $M\unlhd G$ contained in $N$. 

\autoref{Ga} implies that an abelian normal subgroup $N\unlhd G$ has a complement in $G$ if and only if for every Sylow subgroup $P/N$ of $G/N$, $N$ has a complement in $P$. 
This was improved by Šemetkov~\cite[Theorem~2]{Shemetkov} as follows (a very similar result for solvable groups appeared in Wright~\cite[Theorem~2.6]{Wright}).

\begin{Thm}[\textsc{Šemetkov}]\label{Shemetkov}
Let $N\unlhd G$ such that for every prime divisor $p$ of $|G:N|$, $N$ has an abelian Sylow $p$-subgroup $P$, and $P$ has a complement in a Sylow $p$-subgroup of $G$. Then $N$ has a complement in $G$. 
\end{Thm}

Since this result is not very well-known, we provide a proof for the convenience of the reader. Šemetkov's original proof (reproduced in Kirtland's book~\cite[Theorem~4.9]{Kirtland}) is rather involved an depends on a lemma of Huppert. A shorter proof relying on Walter's classification of the simple groups with abelian Sylow $2$-subgroups was given by Baki\'{c}~\cite[Theorem~5]{Bakic}. We combined ideas from both proofs to give a short elementary argument.
The first lemma generalizes the Schur--Zassenhaus theorem (noting that $|H_1|$ divides $|H|$).

\begin{Lem}\label{lem}
Let $N\unlhd G$ and $H\le G$ such that $G=HN$. Then there exists $H_1\le H$ such that $G=H_1N$ and $H_1\cap N\le\Phi(H_1)$. In particular, $|H_1|$ and $|G:N|$ have the same prime divisors.
\end{Lem}
\begin{proof}
Choose $H_1\le H$ minimal with respect to inclusion such that $G=H_1N$. Note that $H_1\cap N$ is normal in $H_1$. Let $M<H_1$ be a maximal subgroup of $H_1$. If $H_1\cap N\nsubseteq M$, then $H_1=M(H_1\cap N)$ and we obtain the contradiction $G=H_1N=MN$. Hence, $H_1\cap N$ is contained in all maximal subgroups of $H_1$ and it follows that $H_1\cap N\le\Phi(H_1)$. In particular, $H_1\cap N$ is nilpotent. Since $G/N\cong H_1/H_1\cap N$, the prime divisors of $|G:N|$ divide $|H_1|$. Conversely, let $p$ be a prime divisor of $|H_1|$. Suppose that $p$ does not divide $|H_1/H_1\cap N|$. Then $H_1\cap N$ contains a Sylow $p$-subgroup $P$ of $H_1$. Since $H_1\cap N$ is nilpotent, it follows that $P$ is the unique Sylow $p$-subgroup of $H_1\cap N\unlhd H_1$ and therefore $P\unlhd H_1$. By the Schur--Zassenhaus theorem, $P$ has a complement $K$ in $H_1$. Now $K$ lies in a maximal subgroup $M<H_1$, but so does $P\le H_1\cap N\le\Phi(H_1)$. Hence, $H_1=PK\le M$, a contradiction. This shows that $p$ divides $|H_1/H_1\cap N|=|G:N|$.
\end{proof}

The following theorem introduces a new parameter in order to perform induction. 

\begin{Thm}\label{thmgeneral}
Let $N\unlhd G$. 
Let $\pi$ be a set of primes $p$ such that there exists a Sylow $p$-subgroup $P$ of $G$ such that $P\cap N$ is abelian and $P\cap N$ has a complement in $P$. Then there exists $H\le G$ such that $G=HN$ and no prime divisor of $|H\cap N|$ lies in $\pi$.
\end{Thm}
\begin{proof}
We argue by induction on $|G|+|N|+|\pi|$. 

\textbf{Case~1:} $N'<N$.\\
Choose a prime $p$ such that $M:=\pcore^p(N)<N$. If $M=1$, then the claim follows with $G=H$ (if $p\notin\pi$) or from \autoref{Ga} (if $p\in\pi$). Thus, let $M\ne 1$ and $\overline{G}:=G/M$. Let $q\in\pi$. By hypothesis, there exists a Sylow $q$-subgroup $G_q=N_q\rtimes R$ of $G$, where $N_q\in\Syl_q(N)$ is abelian. Obviously, $\overline{N_q}\le \overline{G_q}$ are Sylow subgroups $\overline{N}$ and $\overline{G}$ respectively and $\overline{N_q}$ is abelian. Since $R\cap N\le R\cap G_q\cap N=R\cap N_q=1$ we have $N_qM\cap RM=(N_qM\cap R)M=M$ and $\overline{G_q}=\overline{N_q}\rtimes\overline{R}$. By induction there exists $K/M\le \overline{G}$ with $G=KN$ such that $(K\cap N)/M$ is a $\pi'$-group. 

For $q\ne p$, every Sylow $q$-subgroup of $M$ is also a Sylow $q$-subgroup of $N$ and has a complement in $K$. Now let $q=p$ and $K_p\in\Syl_p(K)$. By another theorem of Gaschütz (see \cite[Satz~IV.3.8]{Huppert}), $M=\pcore^p(K_pM)$ has a complement $R$ in $K_pM$.
We have 
\[|R|=|K_pM:M|=|K_p:K_p\cap M|=|K:M|_p.\] 
By Sylow's theorem, $R$ must normalize a Sylow $p$-subgroup $M_p$ of $M$, because their number is $1$ modulo $p$. Now $R$ is a complement of $M_p$ in the Sylow subgroup $M_p\rtimes R$ of $K$. Since $M<N$, the claim holds for $M\le K$ by induction. Thus, there exists $H\le K$ with $K=HM$ such that $H\cap M$ is a $\pi'$-group. Now $G=KN=HMN=HN$. Since 
\[(H\cap N)/(H\cap M)\cong (H\cap N)M/M=(K\cap N)/M,\] 
also $H\cap N$ is a $\pi'$-group.

\textbf{Case~2:} $N=N'$.\\
For $\pi=\varnothing$ the claim holds with $G=H$. Hence, let $p\in\pi$ and $\tau:=\pi\setminus\{p\}$. By induction there exists $K\le G$ with $G=KN$ such that $K\cap N$ is a $\tau'$-group. We may assume that $K\cap N$ is not a $\pi'$-group, i.\,e. $p$ divides $|K\cap N|$. Let $P\in\Syl_p(K\cap N)$. By \autoref{lem}, we may assume that $K\cap N\le\Phi(K)$. In particular, $K\cap N\unlhd K$ is nilpotent and $P=\pcore_p(K\cap N)\unlhd K\le\N_G(P)$. Consider
\[L:=K\C_N(P)\le\N_G(P).\] 
By hypothesis, there exists an abelian Sylow $p$-subgroup $N_p$ of $N$ containing $P$. Then $N_p\in\Syl_p(\C_N(P))$ and $N_p$  has a complement in $L$. Let $q\in\tau$. A Sylow $q$-subgroup $K_q$ of $K$ normalizes some $\C_N(P)_q\in\Syl_q(\C_N(P))$. Since $K\cap N$ is a $\tau'$-group, $K_q\cap\C_N(P)_q=1$ and $\C_N(P)_q\rtimes K_q\in\Syl_q(L)$. By \cite[Satz~IV.2.2]{Huppert} (a generalized version of a theorem of Taunt~\cite[Theorem~4.1]{Taunt}), $N_p\cap\Z(N)=N_p\cap\Z(N)\cap N'=1$. In particular, $\C_N(P)<N$ and we are allowed to apply induction to $\C_N(P)\unlhd L$. This yields $H\le L$ with $L=H\C_N(P)$ such that $H\cap\C_N(P)$ is a $\pi'$-group. 
Then $G=KN=K\N_N(P)N=LN=HN$. Since $|N:\C_N(P)|\not\equiv 0\pmod{p}$,
\[(H\cap N)/(H\cap\C_N(P))\cong(H\cap N)\C_N(P)/\C_N(P)\] 
is a $p'$-group. On the other hand,
\[
(H\cap N)\C_N(P)/\C_N(P)\cong(K\cap N)\C_N(P)/\C_N(P)\cong(K\cap N)/(K\cap\C_N(P))\]
is a $\tau'$-group. Altogether, $H\cap N$ is a $\pi'$-group.
\end{proof}

\begin{proof}[Proof of \autoref{Shemetkov}]
Let $\pi$ be the set of all prime divisors of $|G:N|$. By \autoref{thmgeneral} there exists $H\le G$ with $G=HN$ such that $H\cap N$ is a $\pi'$-group. On the other hand, we may assume that $H\cap N$ is a $\pi$-group by \autoref{lem}. Hence, $H\cap N=1$. 
\end{proof}

We are now in a position to prove our main result.

\begin{proof}[Proof of \autoref{sylow}]
Let $N\le H\le G$ such that $N\unlhd G$, $\gcd(|N|,|G:H|)=1$, and $N$ has a complement $K$ in $H$. Let $p$ be a prime divisor of $|G:N|$. By Sylow's theorem, a given Sylow $p$-subgroup $Q$ of $K$ is contained in a Sylow $p$-subgroup $P$ of $H$. Then $P\cap N\unlhd P$ is a Sylow $p$-subgroup of $N$ and $P=(P\cap N)\rtimes Q$ by comparing orders. By hypothesis, $P\cap N$ is abelian. If $p\nmid |N|$, then $P\cap N=1$ obviously has a complement in a Sylow $p$-subgroup of $G$. On the other hand, if $p$ divides $|N|$, then $p\nmid|G:H|$ and $P$ is a Sylow $p$-subgroup of $G$. Again $P\cap N$ has a complement. 
By Šemetkov's theorem, $N$ has a complement in $G$. 
\end{proof}

Another source of examples to Brandis' question comes from a splitting criterion by Rose~\cite[Corollary~2.3]{Rose} (obtained earlier by Loonstra~\cite[Satz~4.3 and Satz~5.1]{Loonstra} in a less concise form).

\begin{Thm}[\textsc{Rose}]\label{rose}
For every finite group $N$ the following assertions are equivalent:
\begin{enumerate}[(1)]
\item $\Z(N)=1$ and the inner automorphism group $\Inn(N)$ has a complement in $\Aut(N)$.
\item If $N$ is a normal subgroup of some finite group $G$, then $N$ has a complement in $G$.
\end{enumerate}
\end{Thm}

Since Rose's arguments are tailored for infinite groups, we provide a more direct proof.

\begin{proof}
Suppose that (1) holds. Let $N\unlhd G$ and $M:=\C_G(N)\unlhd G$. Then $N\cap M=\Z(N)=1$ and $\Inn(N)\cong N\cong NM/M\unlhd G/M\le\Aut(N)$. Hence, there exists $K/M\le G/M$ such that $G=NK$ and $NM\cap K=M$. It follows that $N\cap K\le N\cap NM\cap K=N\cap M=1$. This shows that $K$ is a complement of $N$ in $G$.

Now assume that (2) is satisfied. By way of contradiction, let $\Z(N)\ne 1$ and choose a prime divisor $p$ of $|\Z(N)|$. 
Let $x\in\Z(N)$ be of order $p$ and let $p^n$ be the maximal order of a $p$-element in $N$. Choose $C=\langle y\rangle\cong C_{p^{n+1}}$ and define 
\[Z:=\langle (x,y^{p^n})\rangle\le\Z(G\times C).\] 
We construct the central product $G:=(N\times C)/Z$. Then the map $f:N\to G$, $g\mapsto(g,1)Z$ is a monomorphism. By hypothesis, $f(N)$ has a complement $K\le G$. By construction, $[f(N),K]=1$ and $G=f(N)\times K$. Since $|K|=|G|/|N|=p^n$, $G$ does not contain elements of order $p^{n+1}$. But $(1,y)Z\in G$ does have order $p^{n+1}$. This contradiction shows that $\Z(N)=1$. Now $\Inn(N)\cong N$ has a complement in $\Aut(G)$ by hypothesis.
\end{proof}

Many (non-abelian) simple groups $N$ satisfy Rose's criterion. For instance, all alternating groups apart from the notable exception $A_6$. The exceptions among the groups of Lie type were classified in \cite{AutSsplit} (see also \cite{Bakic}). 
For centerless perfect groups (i.\,e. $\Z(N)=1$ and $N'=N$) we will show in \autoref{perfect} that Rose's criterion is actually necessary to obtain Gaschütz' theorem.

A group $N$ is called \emph{complete} if it satisfies the stronger condition
\begin{enumerate}
\item[(1')] $\Z(N)=1$ and $\Inn(N)=\Aut(N)$.
\end{enumerate}
In this case $N$ has a unique \emph{normal} complement in $G$ (whenever $N\unlhd G$). In fact, $G=N\times\C_G(N)$. Conversely, a theorem of Baer~\cite[Theorem~1]{BaerComplete} asserts that $N$ is complete if $N$ always has a normal complement in $G$ whenever $N\unlhd G$ (see \cite[Theorems~7.15, 7.17]{RotmanGT}). 

Starting with a centerless group $G$, Wielandt has shown that the \emph{automorphism tower}
\[G\le\Aut(G)\le\Aut(\Aut(G))\le\ldots\]
terminates in a complete group after finitely many steps (see \cite[Theorem~9.10]{IsaacsGroup}). If $G$ is non-abelian simple, then already $\Aut(G)$ is complete according to a result of Burnside (see \cite[Theorem~7.14]{RotmanGT}). In particular, the symmetric groups $S_n\cong\Aut(A_n)$ for $n\ge 7$ are complete (also for $n=3,4,5$ by different reasons).
A large class of complete groups, including some groups of odd order, was constructed in \cite{HartleyRobinson} (a paper dedicated to Gaschütz).  

The following elementary observation extends the class of groups further ($\Out(N)=\Aut(N)/\Inn(N)$ denotes the outer automorphism group of $N$).

\begin{Prop}\label{prop}\hfill
\begin{enumerate}[(i)]
\item Let $N_1,\ldots,N_k$ be finite groups. Then Rose's criterion holds for $N_1\times\ldots\times N_k$ if and only if it holds for $N_1,\ldots,N_k$.

\item\label{three} 
Let $N=N_1\times\ldots\times N_k$ with characteristic subgroups $N_1,\ldots,N_k\le N$. If Gaschütz' theorem holds for $N_1,\ldots,N_k$, then Gaschütz' theorem holds for $N$. 

\item Let $N$ be a finite group with a characteristic subgroup $M$ such that $M$ fulfills Rose's criterion and Gaschütz' theorem holds for $N/M$. Then Gaschütz' theorem holds for $N$.

\item\label{new} Let $N$ be a finite group with a characteristic subgroup $M$ such that $\gcd(|M|,|\Z(N)||\Out(N)|)=1$ and all Sylow subgroups of $M$ are abelian. Suppose that $M$ has a complement in $N$ and Gaschütz' theorem holds for $N/M$. Then Gaschütz' theorem holds for $N$.
\end{enumerate}
\end{Prop}
\begin{proof}\hfill
\begin{enumerate}[(i)]
\item Suppose first that $N_1$ is a normal subgroup of a finite group $G$ such that $N_1$ has no complement in $G$. 
By way of contradiction, suppose that $L$ is a complement of $N_1\times\ldots\times N_k$ in $\hat{G}:=G\times N_2\times\ldots\times N_k$. Let $K:=N_2\ldots N_kL\cap G$. Then
\[N_1K=N_1\ldots N_kL\cap G=\hat{G}\cap G=G\]
and $N_1\cap K=N_1\cap N_2\ldots N_kL=1$, because every element of $\hat{G}$ can be written uniquely as $x_1\ldots x_ky$ with $x_i\in N_i$ and $y\in L$. But now $K$ is a complement of $N_1$ in $G$. Contradiction. Hence, if Rose's criterion holds for $N_1\times\ldots\times N_k$, then it holds for $N_1$ and by symmetry also for $N_1,\ldots,N_k$.

Assume conversely that $N_1,\ldots,N_k$ fulfill Rose's criterion. Then $\Z(N_1\times\ldots\times N_k)=\Z(N_1)\times\ldots\times\Z(N_k)=1$. By the first part of the proof, we may assume that each $N_i$ is indecomposable. Since $\Z(N_1)=\ldots=\Z(N_k)=1$, every automorphism of $N_1\times\ldots\times N_k$ permutes the $N_i$ (see \cite[Satz~I.12.6]{Huppert}). We may arrange the $N_i$ such that 
\[N_1\cong\ldots\cong N_{k_1}\not\cong N_{k_1+1}\cong\ldots\cong N_{k_1+k_2}\not\cong\ldots.\] 
Then $\Aut(N_1\times\ldots\times N_k)\cong\Aut(N_1^{k_1})\times\ldots\times\Aut(N_s^{k_s})$. In order to verify Rose's criterion for $N_1\times\ldots\times N_k$, we may assume that $k_1=k$, i.\,e. $N_1\cong\ldots\cong N_k$. In this case we obtain $\Aut(N_1^k)\cong\Aut(N_1)\wr S_k$. We identify $S_k$ with a subgroup of $\Aut(N_1^k)$. By hypothesis, there exists a complement $K_1$ of $\Inn(N_1)$ in $\Aut(N_1)$. It is easy to see that $\langle K_1,S_k\rangle\cong K_1\wr S_k$ is a complement of $\Inn(N_1^k)$ in $\Aut(N_1^k)$. 

\item Since every automorphism of $N_1\times\ldots\times N_{k-1}$ extends to an automorphism of $N$, it follows that $N_1$ is characteristic in $N_1\times\ldots\times N_{k-1}$. Hence, by induction on $k$, it suffices to consider the case $k=2$. 
Let $N\le H\le G$ such that $N\unlhd G$, $\gcd(|N|,|G:H|)=1$ and $N$ has a complement $K$ in $H$. Then $N_1$ and $N_2$ are normal in $G$, since they are characteristic in $N$. Moreover, $KN_2$ is a complement of $N_1$ in $H$, because 
\[N_1\cap KN_2=N_1\cap N\cap KN_2=N_1\cap(N\cap K)N_2=N_1\cap N_2=1.\]
Since $|N_1|$ is coprime to $|G:H|$, Gaschütz' theorem applied to $N_1$ yields a complement $L$ of $N_1$ in $G$. Now the canonical map $\phi:L\to G/N_1$, $x\mapsto xN_1$ is an isomorphism and we define $L_N:=\phi^{-1}(N/N_1)$ and $L_H:=\phi^{-1}(H/N_1)$. Since $KN_1/N_1$ is a complement of $N/N_1$ in $H/N_1$, also $L_N$ has a complement in $L_H$. Moreover, $|L:L_H|=|G:H|$ is coprime to $|L_N|=|N/N_1|$. Now Gaschütz' theorem applied to $L_N\cong N/N_1\cong N_2$ provides a complement $L_K$ of $L_N$ in $L$. Then $L_K\cap N\le L_K\cap L_N=1$ and $|L_K|=|L:L_N|=|G:N|$. Therefore, $L_K$ is a complement of $N$ in $G$. 

\item\label{four} Let $N\le H=N\rtimes K\le G$ as usual. Since $M$ is characteristic in $N$, we have $M\unlhd G$ and $KM/M$ is a complement of $N/M\unlhd H/M$. By Gaschütz' theorem, $N/M$ has a complement $L/M$ in $G/M$. By Rose's theorem, $M$ has a complement $\hat{K}$ in $L$. Now $G=LN=\hat{K}MN=\hat{K}N$ and 
\[\hat{K}\cap N=\hat{K}\cap L\cap N=\hat{K}\cap M=1.\]
Therefore, $\hat{K}$ is a complement of $N$ in $G$.

\item Let $N\le H\le G$ as usual. As in \eqref{four} we find $L\le G$ such that $G=NL$ and $N\cap L=M$. It suffices to show that $M$ has a complement in $L$. We do this using Šemetkov's theorem. Let $P$ be a non-trivial Sylow $p$-subgroup of $M$. 
Let $Q$ be a Sylow $p$-subgroup of a complement of $M$ in $N$ (which exists by hypothesis). By Sylow's theorem, we may assume that $Q$ normalizes $P$, so that $P\rtimes Q$ is a Sylow $p$-subgroup of $N$. Let $R$ be a Sylow $p$-subgroup of $\C_G(N)$. Then $QR$ is a $p$-subgroup of $N\C_G(N)$.
Let $x=st\in P\cap QR$ with $s\in Q$ and $t\in R$. Then $t=s^{-1}x\in N\cap\C_G(N)=\Z(N)$. Since $\Z(N)$ is a $p'$-group by hypothesis, we obtain $t=1$ and $x=s\in P\cap Q=1$. This shows that $P\cap QR=1$. Since $G/N\C_G(N)\le\Out(N)$ and $\Out(N)$ is a $p'$-group, $P\rtimes QR$ is a Sylow $p$-subgroup of $G$. 
Hence, $P$ also has a complement in a Sylow $p$-subgroup of $L$. Since $P$ is abelian, Šemetkov's theorem applies to $M$.
\qedhere
\end{enumerate}
\end{proof}

For instance, if all non-abelian minimal normal subgroups $M_1,\ldots,M_n$ of $N$ fulfill Rose's criterion and if all Sylow subgroups of $N/M_1\ldots M_n$ are abelian, then Gaschütz' theorem holds for $N$ by \autoref{sylow} and \autoref{prop}.

Using \cite[Satz~I.12.6]{Huppert}, it is easy to see that $N_1,\ldots,N_k$ are characteristic in $N=N_1\times\ldots\times N_k$ if and only if the following holds for all $i\ne j$:
\begin{enumerate}[(i)]
\item $N_i$ and $N_j$ have no common direct factor,
\item $\gcd\bigl(|N_i/N_i'|,|\Z(N_j)|\bigr)=1$.
\end{enumerate}

Concrete examples for \autoref{prop}\eqref{new} are groups of the form $N=P\rtimes Q$ where $P$ and $Q$ are abelian of order $9$ and $12$ respectively and $|\Z(N)|=2$ (there are four isomorphism types for $N$). Now \autoref{prop}\eqref{three} applies to $N\times C_7$ (while the other parts do not apply here). 

We now recall a theorem of Carter~\cite[Theorem~4]{CarterSplitting}, which generalizes work of Higman~\cite[Theorem~3]{HigmanSplitting}, Schenkman~\cite[Theorem~1]{SchenkmanSplitting} and Yonaha~\cite{Yonaha} (there is an even more general version by Shult~\cite{Shult} in terms of formations, which is presented in Huppert's book~\cite{Huppert}). 
Let $G^\infty$ be the \emph{nilpotent residue} of $G$, i.\,e. $G^\infty$ is the smallest normal subgroup of $G$ with nilpotent quotient $G/G^\infty$. Set $L_0(G):=G$ and $L_{n+1}(G):=L_n(G)^\infty$ for $n\ge 0$. Then $G=L_0(G)\ge L_1(G)\ge\ldots$ is sometimes called the \emph{lower nilpotent series} of $G$. Note that $G$ is solvable if and only if $L_n(G)=1$ for some $n$. 

\begin{Thm}[\textsc{Carter}]\label{carter}
Suppose that $L_n(G)$ is abelian for some $n\ge 0$. Then $L_n(G)$ has a complement in $G$ and all complements are conjugate. 
\end{Thm}

For our purpose we also require that $G/L_n(G)$ is abelian (hence $n\le 1$). 
Recall that a group $G$ is called \emph{metabelian} if $G'$ is abelian, i.\,e. $G''=1$. 

\begin{Cor}[\textsc{Yonaha}]\label{yonaha}
Let $G$ be a metabelian group such that $\Z(G)\cap G'=1$. Then $G'$ has a complement in $G$ and all such complements are conjugate.
\end{Cor}
\begin{proof}
By \autoref{carter}, it suffices to show that $G'=L_1(G)=G^\infty$. Clearly, $G^\infty\le G'$. Let $H\le G$ minimal such that $G=G^\infty H$. Then $H\cap G^\infty\le\Phi(H)$ by \autoref{lem}. Since $G/(H\cap G^\infty)\cong HG^\infty/G^\infty=G/G^\infty$ is nilpotent, it follows that $H$ is nilpotent (see \cite[Lemma~9.19]{IsaacsGroup}). Suppose that $N:=H\cap G'\ne 1$. Since $N\unlhd H$, also $M:=N\cap\Z(H)\ne 1$. Now $M$ is centralized by $H$ and by $G'$ since $G'$ is abelian. Hence, $M$ is centralized by $HG'=G$. This yields the contradiction $M\le G'\cap\Z(G)=1$. Therefore, $H\cap G'=1=H\cap G^\infty$ and $|G'|=|G'H:H|=|G:H|=|G^\infty H:H|=|G^\infty|$.
This shows not only that $G^\infty=G'$, but also that $G'$ has a complement. For the uniqueness we refer to \autoref{carter}.
\end{proof}

We remark that, conversely, if $G^\infty$ is abelian, then $G^\infty\cap\Z(G)=1$ by \cite[Theorem~1]{CarterSplitting}.

\begin{Thm}\label{metabel}
If $N$ is metabelian and $N'\cap\Z(N)=1$, then Gaschütz' theorem holds for $N$.
\end{Thm}
\begin{proof} 
Suppose that $N\le H\le G$ such that $N\unlhd G$, $\gcd(|N|,|G:H|)=1$ and $N$ has a complement in $H$.
By \autoref{yonaha}, $M:=N'$ has a complement $K$ in $N$ and all such complements are conjugate in $N$. The Frattini argument implies that $G=N\N_G(K)=M\N_G(K)$. 
For $x\in M\cap\N_G(K)$ and $y\in K$ we have $[x,y]\in M\cap K=1$. Hence, 
\[x\in M\cap\C_N(K)=N'\cap\Z(N)=1,\] 
because $M$ is abelian. Therefore, $\N_G(K)$ is a complement of $M$ in $G$. 
Since $N/M$ is abelian, Gaschütz' theorem applied to $N/M\le H/M\le G/M$ yields $G=NL$ with $N\cap L=M$ as usual. As in the proof of \autoref{prop}, it suffices to show that $M$ has a complement in $L$. But this is clear, since $M$ has a complement in $G$.
\end{proof}

Surely the proof of \autoref{metabel} can be adapted to similar situations (using the Schur--Zassenhaus theorem instead of \autoref{yonaha} for instance).

For the sake of completeness we also address the dual of Rose's theorem which is probably known to experts in cohomology.

\begin{Thm}\label{dual}
For every finite group $K\ne 1$ there exist finite groups $N\unlhd G$ such that $G/N\cong K$ and $N$ has no complement in $G$.
\end{Thm}
\begin{proof}
Again it was Gaschütz~\cite{GaschutzFrattini} who proved a stronger statement where $N$ is required to lie in the Frattini subgroup of $G$ (then $G$ is called a \emph{Frattini extension} of $K$). The following arguments are inspired by \cite[Theorem~B.11.8]{DoerkHawkes}. (A cohomological proof can be given with Shapiro's lemma, see \cite[Proposition 9.76]{RotmanHom}.) Let $K=F/R$ where $F$ is a free group of finite rank and $R\unlhd F$. Let $P/R\le F/R$ be a subgroup of prime order $p$ (exists since $K\ne 1$). By the Nielsen--Schreier theorem, $P$ is free and $P/P'$ is free abelian of finite rank. Therefore we find $P_1\unlhd P$ with $P_1\le R$ and $P/P_1\cong C_{p^2}$. Let $Q\unlhd F$ be the kernel of the permutation action of $F$ on the cosets $F/P_1$. Then $Q\le P_1$ and $|F:Q|\le|F:P_1|!<\infty$. 

Define $G:=F/Q$ and $N:=R/Q$. Clearly, $G/N\cong F/R\cong K$. Suppose that $N$ has a complement $H/Q$ in $G$. Then $(H\cap P)/Q$ is a complement of $N$ in $P/Q$. Moreover, $(H\cap P)P_1/P_1$ is a complement of $R/P_1$ in $P/P_1$. But this is impossible since $P/P_1\cong C_{p^2}$. 
\end{proof}

The situation of \autoref{dual} is different for infinite groups: Every group $K$ is a quotient of a free group $F$. If $F$ splits, then $K$ is a subgroup of $F$ and therefore free by the Nielsen--Schreier theorem. Conversely, by the universal property of free groups, every group extension with a free quotient $K$ splits (including the case $K=1$).

We use the opportunity to mention a result in the opposite direction by Gaschütz and Eick~\cite{EickGa}:
\begin{Thm}[\textsc{Gaschütz}, \textsc{Eick}]
For a finite group $N$ the following assertions are equivalent:
\begin{enumerate}[(i)]
\item There exists a finite group $G$ with $N\unlhd G$ such that $NH<G$ for all $H<G$.
\item There exists a finite group $G$ with $N=\Phi(G)$.
\item $\Inn(N)\le\Phi(\Aut(N))$.
\end{enumerate}
\end{Thm}

Many more complement theorems can be found in Kirtland's recent book~\cite{Kirtland}.

\section{Some non-existence theorems}\label{sec3}

In the proof of \autoref{thmgeneral} we have already mentioned \cite[Theorem~IV.2.2]{Huppert}, which implies that
\[\Z(G)\cap G'\cap P\le P'\] 
for every finite group $G$ with Sylow subgroup $P$. Hence, in the situation of \autoref{sylow} we have $\Z(G)\cap G'=1$. Our main theorem shows that this is in fact a necessary condition for Gaschütz' theorem.
Recall from the introduction that Gaschütz' theorem holds for $N$ if for every embedding $N\le H\le G$ such that $N\unlhd G$, $\gcd(|N|,|G:H|)=1$ and $N$ has a complement in $H$, then $N$ has a complement in $G$.

\begin{Thm}\label{ZNthm}
Let $N$ be a finite group such that $\Z(N)\cap N'\ne 1$. Then for every integer $q>1$ coprime to $|N|$ there exist groups $N\le H\le G$ with the following properties: 
\begin{enumerate}[(i)]
\item $N\unlhd G$ and $H\unlhd G$.
\item $N$ has a complement in $H$, but not in $G$.
\item $H$ and $N$ have the same composition factors (up to multiplicities) and $G/H$ is cyclic of order $q$.
\end{enumerate}
In particular, Gaschütz' theorem does not hold for $N$.
\end{Thm}
\begin{proof}
Let $1\ne Z=\langle z\rangle\le\Z(N)\cap N'$. Let $\alpha$ be the automorphism of $D:=N^q=N\times\ldots\times N$ such that $\alpha(x_1,\ldots,x_q)=(x_q,x_1,\ldots,x_{q-1})$ for all $(x_1,\ldots,x_q)\in D$. Then $W:=D\rtimes\langle\alpha\rangle\cong N\wr C_q$ and $\overline{z}:=(z,\ldots,z)\in\Z(W)$. Hence, we can construct the central product
\[G:=(N\times W)/\langle(z,\overline{z})\rangle\cong N*W.\]
We identify $N$, $D$ and $W$ with their images in $G$. In this sense, $N\cap W=N\cap D=\langle \overline{z}\rangle=\langle z^{-1}\rangle$. 
Now $H:=ND\unlhd G$ has the same composition factors as $N$ and $G/H\cong C_q$. 
Consider 
\[K:=\{x_1(x_1,\ldots,x_q):(x_1,\ldots,x_q)\in D\}\le H.\] 
Clearly, $H=NK$. For $g=x_1(x_1,\ldots,x_q)\in K\cap N$ we must have $(x_1,\ldots,x_q)\in\langle\overline{z}\rangle$ and therefore $g=1$. Hence, $K$ is a complement of $N$ in $H$.

Suppose by way of contradiction that $N$ has a complement $L$ in $G$. Note that $\langle\alpha\rangle$ is a nilpotent Hall subgroup of $G$. A theorem of Wielandt asserts that every Hall subgroup of order $q$ is conjugate to $\langle\alpha\rangle$ (see \cite[Satz~III.5.8]{Huppert}; if $q$ is a prime, Sylow's theorem suffices).
Since every conjugate of $L$ in $G$ is also a complement of $N$, we may assume that $\alpha\in L$. It follows that $L\cap H$ is an $\alpha$-invariant complement of $N$ in $H$. 
For every $d\in D$ there exists $x\in N$ such that $xd\in L$. Consequently, $\alpha(d)d^{-1}=\alpha(xd)(xd)^{-1}\in L$. 
In particular, $(x,x^{-1},1,1\ldots,1)\in L$ for all $x\in N$. For $x,y\in N$ we compute
\begin{equation}\label{comminL}
([x,y],1,\ldots,1)=(x,x^{-1},1,\ldots,1)(y,y^{-1},1,\ldots,1)((yx)^{-1},yx,1,\ldots,1)\in L.
\end{equation}
Since $z\in N'$, we conclude that $(z,1,\ldots,1)\in L$. But now also 
\[\overline{z}=(z,1,\ldots,1)\alpha(z,1,\ldots,1)\ldots\alpha^{q-1}(z,1,\ldots,1)\in L\cap N.\]
This contradicts $L\cap N=1$. 
\end{proof}

\begin{Cor}
Gaschütz' theorem fails for all non-abelian nilpotent groups. 
\end{Cor}
\begin{proof}
See \cite[Satz~III.2.6]{Huppert} for instance.
\end{proof}

\begin{Cor}
If $N$ is metabelian, then Gaschütz' theorem holds for $N$ if and only if $N'\cap\Z(N)=1$.
\end{Cor}
\begin{proof}
This follows from \autoref{metabel}.
\end{proof}

We illustrate that the condition $\Z(N)\cap N'=1$ (even $\Z(N)=1$) is not sufficient for Gaschütz' theorem in general. 
A given counterexample $N\unlhd H\le G$ to Gaschütz' theorem can be “blown up” as follows. 
Let $L$ be a finite group such that $\gcd(|L|,|G:H|)=1$ (this is a harmless restriction in the situation of \autoref{ZNthm}). 
To an arbitrary homomorphism $G\to\Aut(L)$, we form the semidirect products $\hat G:=L\rtimes G$, $\hat H:=L\rtimes H$ and $\hat N:=L\rtimes N$. If $K$ is a complement of $N$ in $H$, then $K$ is also a complement of $\hat N$ in $\hat H$. Now suppose that $\hat K$ is a complement of $\hat N$ in $\hat G$. Then $\hat KL/L$ a complement of $\hat N/L\cong N$ in $\hat G/L\cong G$. Contradiction. Hence, $\hat N\unlhd \hat H\le\hat G$ is a counterexample to Gaschütz' theorem. 

The counterexample $\SL(2,3)*C_4$ mentioned in the introduction lives inside $\GL(2,5)$. Therefore, Gaschütz' theorem does not hold for the Frobenius group $N=C_5^2\rtimes Q_8$. Indeed, $\Z(N)=1$. In contrast, Gaschütz' theorem does hold the very similar groups $C_5^2\rtimes D_8$ and $C_3^2\rtimes Q_8$, because those fulfill Rose's criterion. So we see that the question for an individual group can be very delicate to answer. 

Other examples arise from our next theorem, which is related to Rose's result as well.

\begin{Thm}\label{perfect}
Let $N$ be a perfect group with trivial center. Then Gaschütz' theorem holds for $N$ if and only if $\Inn(N)$ has a complement in $\Aut(N)$. 
\end{Thm}
\begin{proof}
If $\Inn(N)$ has a complement in $\Aut(N)$, then the claim follows from \autoref{rose}. Now assume conversely that $\Inn(N)$ has no complement in $\Aut(N)$. We construct a counterexample similar as in \autoref{ZNthm}. Since $\Z(N)=1$, we will identify $N$ with $\Inn(N)$. Let $q>1$ be an integer coprime to $|\Aut(N)|$. Let $\alpha$ be the automorphism of $D:=N^q=N\times\ldots\times N$ such that $\alpha(x_1,\ldots,x_q)=(x_q,x_1,\ldots,x_{q-1})$ for all $(x_1,\ldots,x_q)\in D$. Then $W:=D\rtimes\langle\alpha\rangle\cong N\wr C_q$ is a subgroup of $\Aut(N\times D)\cong\Aut(N)\wr S_{q+1}$. Since the diagonal subgroup $A:=\langle(\gamma,\ldots,\gamma):\gamma\in\Aut(N)\rangle\le\Aut(N)^{q+1}$ is centralized by $\alpha$, we can define $G:=NWA$ and $H:=NDA\unlhd G$. As usual, we identify $N$, $D$, $W$ and $A$ with subgroups of $G$. We show that Gaschütz' theorem fails with respect to $N\le H\le G$.

Note first that $G/H\cong\langle\alpha\rangle\cong C_q$. As in the proof of \autoref{ZNthm}, it is easy to see that
\[\langle x_1(x_1,x_2,\ldots,x_q):x_1,\ldots,x_q\in N\rangle A\le H\]
is a complement of $N$ in $H$. Suppose by way of contradiction that $L\le G$ is a complement of $N$ in $G$. By Wielandt's theorem on nilpotent Hall subgroups, we may assume that $\alpha\in L$. The same computation as in \eqref{comminL} shows that $D'\le L$. Since $N'=N$ by hypothesis, it follows that $C:=\C_G(N)=D\langle\alpha\rangle\le L$. But now $L/C$ is a complement of $NC/C\cong N$ in $G/C\cong\Aut(N)$.
Contradiction.
\end{proof}

As promised earlier, \autoref{perfect} implies that Gaschütz' theorem does not hold for the alternating group $A_6$. 

\begin{Cor}
Let $N$ be a perfect group with trivial center such that $\Inn(N)$ has no complement in $\Aut(N)$. Then for every finite group $M$, Gaschütz' theorem does not hold for $N\times M$.
\end{Cor}
\begin{proof}
Let $N\le H\le G$ be the counterexample for $N$ constructed in the proof of \autoref{perfect} with $q$ coprime to $|M|$. Then $N\times M\le H\times M\le G\times M$ is a counterexample to Gaschütz' theorem for $N\times M$. 
\end{proof}

An easy variant of \autoref{perfect} yields the following more technical criterion.

\begin{Prop}\label{special}
Let $N$ be a finite group with $\Z(N)=1$. Suppose that there exist $k\in\NN$ and an automorphism $\gamma\in\Aut(N)$ such that $\gamma^k\in\Inn(N)'$ and $(\delta\gamma)^k\ne 1$ for all $\delta\in\Inn(N)$. Then Gaschütz' theorem does not hold for $N$.
\end{Prop}
\begin{proof}
Let $G=NWA$ be the group constructed in the proof of \autoref{perfect} (this does not require $N'=N$). Suppose that $L$ is a complement of $N$ in $G$. Then $D'\le L$ as shown by a computation as in \eqref{comminL}. Since $(\gamma,\ldots,\gamma)\in A\le G=NL$, there exists $\delta\in N$ such that $\delta(\gamma,\ldots,\gamma)\in L$. It follows that $(\delta\gamma)^k(\gamma^k,\ldots,\gamma^k)\in L$. Since $\gamma^k\in N'$, we obtain $(\delta\gamma)^k\in N\cap L=1$. Contradiction.
\end{proof}

\autoref{special} applies for instance to the non-perfect group $N=S_3\wr C_2$ (here $\Aut(N)\cong C_3^2\rtimes SD_{16}$ where $SD_{16}$ is the semidihedral group of order $16$). 

Since there is a gap between the existence theorems in \autoref{sec2} and the non-existence theorem above, it is of interest to look at small examples. Using GAP~\cite{GAP48}, we were able to decide for every group of order less than $144$ whether Gaschütz' theorem holds. We put the first open case as a problem for future research.

\begin{Prob}
Let $N:=(C_3^2\rtimes Q_8)\times C_2=\mathtt{SmallGroup}(144,187)$ with $|\Z(N)|=2$. Decide whether or not Gaschütz' theorem holds for $N$.
\end{Prob}

\section*{Acknowledgment}
\autoref{prop}\eqref{three} was found by Scheima Obeidi within the framework of her Master's thesis written under the direction of the author. I thank Stefanos Aivazidis for some comments on Yonaha's theorem.
I appreciate some valuable comments of an anonymous referee.
The work is supported by the German Research Foundation (\mbox{SA 2864/1-2} and \mbox{SA 2864/4-1}).

\end{document}